\documentclass[11pt]{article}

\input{style}

\begin{document}

\title{Existence of an endogenously complete equilibrium driven by a
  diffusion}

\author{Dmitry Kramkov\thanks{This research was supported in part by
    the Carnegie Mellon-Portugal Program and by the Oxford-Man
    Institute for Quantitative Finance at the University of Oxford.}\\
  Carnegie Mellon University and
  University of Oxford,\\
  Department of Mathematical Sciences,\\ 5000 Forbes Avenue,
  Pittsburgh, PA, 15213-3890, USA}

\date{\today}

\maketitle
\begin{abstract}
  The existence of complete Radner equilibria is established in an
  economy whose parameters are driven by a diffusion process. Our
  results complement those in the literature. In particular, we work
  under essentially minimal regularity conditions and treat the
  time-inhomogeneous case.
\end{abstract}

\begin{description}
\item[MSC:] 91B50, 91B51, 60G44, 26E05. 
\item[Keywords:] real-analytic functions, dynamic equilibrium,
  endogenous completeness, martingale representation.
\end{description}

\section{Introduction}
\label{sec:introduction}

The three basic topics in asset pricing theory are arbitrage,
single-agent optimality, and equilibrium; see, e.g., \citet{Duff:01},
\citet{DanJeanb:03}, and \citet{KaratShr:98}. While in the first two
cases the general mathematical theory is essentially complete; see,
e.g., \citet{DelbSch:94} (arbitrage) and \citet{KramSch:03} and
\citet{Most:14} (optimal investment), the situation with equilibrium
is more involved.

In the continuous time setting, when stocks are defined by their
dividends, the first sufficient conditions for the existence of
dynamic Radner equilibria have been obtained only recently in
\citet{AnderRaim:08}.  A key idea is to look for a \emph{complete}
Radner equilibrium, which is then necessarily constructed in two
steps.  First, one obtains a static Arrow-Debreu equilibrium; this is
standard and requires very few assumptions; see, e.g., \citet{Dana:93}
and \citet{Kram:13}. Then one proves the \emph{endogenous
  completeness} property, that is, shows that the stock prices
generated from the dividends by this Arrow-Debreu equilibrium define a
complete financial market.

The novelty and difficulty are in the second step, which, from a pure
mathematical point of view, can be stated as the following
\emph{backward martingale representation problem};
see~\citet{KramPred:14}. There are random variables $\zeta>0$ and
$\psi = (\psi^j)$, which are explicitly constructed in terms of
model's primitives: dividends, incomes, and utilities. A probability
measure $\mathbb{Q}$ and a $\mathbb{Q}$-martingale $S = (S^j)$ are
defined by
\begin{displaymath}
  \frac{d\mathbb{Q}}{d\mathbb{P}} = \const \zeta \quad \text{and}
  \quad 
  S^j_t = \mathbb{E}_{\mathbb{Q}}[\psi^j|\mathcal{F}_t].
\end{displaymath}
The problem is to check whether every $\mathbb{Q}$-martingale can be
written as a stochastic integral with respect to $S$. In a diffusion
framework, this reduces to the verification of the full rank property
of the Jacobian matrix for the solution of a system of uncoupled
linear parabolic PDEs; see \citet{Takac:12} and \citet{KannRaim-13}.

The results of \citet{AnderRaim:08} have been generalized in
\citet{HugMalTrub:12} and~\citet{RiedHerz:13}. We build on these
works.  Similar to these papers we assume that the parameters of the
economy have a Markov-type dependence on a diffusion process, that the
Jacobian matrix of terminal dividends is non-degenerated, and that the
dependence on time is analytic.

Our contribution is two-fold. First, we work under essentially minimal
regularity conditions. For instance, with respect to the state
variable, the volatility matrix of underlying diffusion is only
uniformly continuous, while the drift vector is just measurable.
Second, we rigorously treat the case of time-dependence in dividends,
income streams, and utilities. For instance, we allow the utility
functions for intermediate consumption to be of the form: $e^{-\nu^m
  t} u_m(c)$ with \emph{different} impatience coefficients $\nu^m$. A
detailed comparison between our assumptions and those in the
literature is given in Section~\ref{sec:model-primitives}.

\subsection*{Basic concepts and notations}
\label{sec:notations}

If $x$ and $y$ are vectors in $\mathbb{R}^n$, then $\ip{x}{y}$ denotes
the scalar product and $\abs{x} \set \sqrt{\ip{x}{x}}$. If $a\in
\mathbb{R}^{m\times n}$ is a matrix with $m$ rows and $n$ columns,
then $ax$ denotes its product on the (column-)vector $x$, $a^*$ stands
for the transpose, and $\abs{a} \set \sqrt{\trace(aa^*)}$.

Let $\mathbf{X}$ be a Banach space and $D$ be a set in a Euclidean
space $\mathbb{R}^d$ contained in the closure of its interior. A map
$\map{f}{D}{\mathbf{X}}$ is \emph{analytic} if for every $x\in D$
there exist a number $\epsilon(x)>0$ and elements $(A_\alpha(x))$ in
$\mathbf{X}$ such that
\begin{displaymath}
  f(y) = \sum_{\alpha} A_\alpha(x) (y-x)^\alpha, \; y\in D,
  \abs{y-x}<\epsilon(x). 
\end{displaymath}
Here the series converges in the norm $\norm{\cdot}_{\mathbf{X}}$ of
$\mathbf{X}$, the summation is taken with respect to multi-indices
$\alpha = (\alpha_1, \dots,\alpha_d)$ of non-negative integers, and,
for $x = (x_1,\dots,x_d)$, $x^\alpha \set \prod_{i=1}^d
x_i^{\alpha_i}$.  A map $\map{f}{D}{\mathbf{X}}$ is \emph{H\"older
  continuous} if there is $0<\delta<1$ such that
\begin{displaymath}
  \sup_{x,y\in D, \; x\not=y}
  \frac{\norm{f(x)-f(y)}_X}{\abs{x-y}^\delta} < \infty.
\end{displaymath}
Of course, a map $\map{f}{D}{\mathbf{X}}$ is \emph{bounded} if $
\sup_{x\in D} \norm{f(x)}_X< \infty$.

In this paper, $\mathbf{X}$ is one of the following spaces of
functions defined on a set $E\subset \mathbb{R}^d$, which is a
$F_\sigma$-set, that is, a countable union of closed sets:
\begin{description}
\item[$\mathbf{L}_\infty=\mathbf{L}_\infty(E) =
  \mathbf{L}_\infty(E,dx)$\rm{:}] the Lebesgue space of bounded
  real-valued functions $f$ on $E$ with the norm
  $\norm{f}_{\mathbf{L}_\infty} \set \esssup_{x\in E} \abs{f(x)}$.
\item[$\mathbf{C}=\mathbf{C}(E)$\rm{:}] the Banach space of bounded
  and continuous real-valued functions $f$ on $E$ with the norm
  $\norm{f}_{\mathbf{C}} \set \sup_{x\in E} \abs{f(x)}$.
\end{description}

\section{Complete Radner equilibrium}
\label{sec:complete-equilibrium}

We consider an economy with $M$ agents and a single consumption good.
The uncertainty and the information flow are modeled by a complete
filtered probability space $(\Omega, \mathcal{F}_1, \mathbf{F} =
(\mathcal{F}_t)_{t\in [0,1]}, \mathbb{P})$.  The agents choose
cumulative consumption processes $C = (C_t)_{t\in [0,1]}$ from a set
$\mathcal{C}$ of optional non-decreasing processes. Their preferences
regarding consumption are specified in terms of (expected) utility
functionals:
\begin{displaymath}
  \map{\mathbb{U}^m}{\mathcal{C}}{\mathbb{R}\cup \braces{-\infty}} \cup
  \braces{\infty}, \; m=1,\dots,M.
\end{displaymath}
The agents receive cumulative income processes $I^m \in \mathcal{C}$,
$m=1,\dots,M$.

The financial market consists of a num\'eraire and $J$ stocks. All
payments are made in the consumption units. The num\'eraire pays the
notional $\Psi>0$ at maturity $t=1$; for instance, if $\Psi=1$, then
the num\'eraire is the zero-coupon bond. The stocks pay the dividend
rates $\theta=(\theta^j_t)$ and the terminal dividends
$\Theta=(\Theta^j)$. Thus, the total dividend paid by $j$th stock up
to time $t\in [0,1]$ is
\begin{displaymath}
  D^j_t = \int_0^t \theta^j_u du + \Theta^j \ind{t=1}. 
\end{displaymath}
The prices of the assets are determined \emph{endogenously} by the
equilibrium mechanism specified in Definition~\ref{def:1}.

By a \emph{$(B,S)$-market} we call an optional process $B=(B_t)>0$ and
a $J$-dimensional semimartingale $S=(S^j_t)$ having the terminal
values
\begin{equation}
  \label{eq:1}
  B_1 = \Psi \quad \text{and} \quad S^j_1 = 
  \frac{\Theta^j}{B_1} + \int_0^1\frac{\theta^j_u}{B_u} du, \; j=1,\dots,J. 
\end{equation}
Here $B$ denotes the price process of the num\'eraire and $S^jB$
represents the wealth process of the buy-and-hold strategy for $j$th
stock; equivalently, $S^j$ is the \emph{discounted} value of this
strategy.

A probability measure $\mathbb{Q}$ is an \emph{equivalent martingale
  measure for $S$} if $\mathbb{Q}\sim \mathbb{P}$ and $S$ is a
$\mathbb{Q}$-martingale.  We call a $(B,S)$-market \emph{complete} if
there is only one such $\mathbb{Q}$; this is the case if and only if
every local martingale under $\mathbb{Q}$ is a stochastic integral
with respect to $S$, see \citet[Section XI.1(a)]{Jacod:79}.

We now introduce the main object of our study.

\begin{Definition}
  \label{def:1}
  A pair $((B,S), (\widehat C^m)_{m=1,\dots,M})$, consisting of a
  $(B,S)$-market and consumptions $\widehat C^m \in \mathcal{C}$,
  $m=1,\dots,M$, is a \emph{complete Radner equilibrium} if
  \begin{enumerate}
  \item The $(B,S)$-market is complete; denote by $\mathbb{Q}$ the
    unique equivalent martingale measure for $S$.
  \item The \emph{clearing condition} holds:
    \begin{equation}
      \label{eq:2}
      \sum_{m=1}^M \widehat C^m_t = \sum_{m=1}^M I^m_t, \; t\in
      [0,1].
    \end{equation}
  \item For every $m=1,\dots,M$ we have
    \begin{align*}
      \abs{\mathbb{U}^m(\widehat C^m)} +
      \mathbb{E}^\mathbb{Q}[\int_0^1 \frac{dI^m_t}{B_t}] < \infty,
    \end{align*}
    the consumption $\widehat C^m$ satisfies the \emph{budget
      constraint}:
    \begin{displaymath}
      \mathbb{E}^\mathbb{Q}[\int_0^1 \frac{d\widehat C^m_t}{B_t}] =
      \mathbb{E}^\mathbb{Q}[\int_0^1 \frac{dI^m_t}{B_t}], 
    \end{displaymath}
    and $\mathbb{U}^m(\widehat C^m)\geq \mathbb{U}^m(C)$ for every
    consumption $C\in \mathcal{C}$ satisfying same budget constraint:
    \begin{displaymath}
      \mathbb{E}^\mathbb{Q}[\int_0^1 \frac{dC_t}{B_t}] =
      \mathbb{E}^\mathbb{Q}[\int_0^1 \frac{dI^m_t}{B_t}]. 
    \end{displaymath}
  \end{enumerate}
\end{Definition}

\begin{Remark}
  \label{rem:1}
  Under Radner equilibrium, for $m$th agent, the transfer of the
  income $I^m$ into the optimal consumption $\widehat C^m$ is
  accomplished by the dynamic trading in stocks. The predictable
  process $\widehat H^m$ of the number of stocks is determined by the
  integral representation:
  \begin{displaymath}
    \int_0^t \widehat H^m_u dS_u = \mathbb{E}^{\mathbb{Q}}[\int_0^1
    \frac{dI^m_u - d\widehat C^m_u}{B_u}|\mathcal{F}_t], \quad t\in
    [0,1].  
  \end{displaymath}
  Observe that the clearing condition~\eqref{eq:2} for consumptions
  yields the clearing condition for stocks:
  \begin{displaymath}
    \sum_{m=1}^M \widehat H^m_t = 0, \; t\in  [0,1].  
  \end{displaymath}
  This condition is usually a part of the definition of Radner
  equilibrium; see, e.g., Definition~5.1 in Section~4.5 of
  \cite{KaratShr:98}.
\end{Remark}

For the convenience of future references we also recall the definition
of the (static) Arrow-Debreu equilibrium.

\begin{Definition}
  \label{def:2}
  A pair $(P, (\widehat C^m)_{m=1,\dots,M})$, consisting of a positive
  optional \emph{consumption price} process $P>0$ and consumptions
  $\widehat C^m \in \mathcal{C}$, $m=1,\dots,M$, is an
  \emph{Arrow-Debreu equilibrium} if the clearing
  condition~\eqref{eq:2} holds and for every $m=1,\dots,M$ we have
  \begin{align*}
    \abs{\mathbb{U}^m(\widehat C^m)} + \mathbb{E}[\int_0^1
    P_t{dI^m_t}] < \infty,
  \end{align*}
  the consumption $\widehat C^m$ satisfies the \emph{budget
    constraint}:
  \begin{displaymath}
    \mathbb{E}[\int_0^1 P_t d\widehat C^m_t] =
    \mathbb{E}[\int_0^1 P_t dI^m_t], 
  \end{displaymath}
  and $\mathbb{U}^m(\widehat C^m)\geq \mathbb{U}^m(C)$ for every
  consumption $C\in \mathcal{C}$ satisfying same budget constraint:
  \begin{displaymath}
    \mathbb{E}[\int_0^1 P_tdC_t] =
    \mathbb{E}[\int_0^1 P_tdI^m_t]. 
  \end{displaymath}
\end{Definition}

The relations between these two types of equilibria are summarized in
the following lemmas.

\begin{Lemma}
  \label{lem:1}
  Let $((B,S),(\widehat C^m)_{m=1,\dots,M})$ be a complete Radner
  equilibrium. Let $\mathbb{Q}$ be the equivalent martingale measure
  for $S$, denote by $Z$ the density process of $\mathbb{Q}$ under
  $\mathbb{P}$, and set $P \set Z/B$. Then $\mathbb{E}[P_1\Psi] =1$
  and the pair $(P,(\widehat C^m)_{m=1,\dots,M})$ is an Arrow-Debreu
  equilibrium.
\end{Lemma}

\begin{Lemma}
  \label{lem:2}
  Let $(P,(\widehat C^m)_{m=1,\dots,M})$ be an Arrow-Debreu
  equilibrium such that $\mathbb{E}[P_1\Psi] =1$. Let $\mathbb{Q}$ be
  the probability measure with the density process $Z_t \set
  \mathbb{E}[P_1\Psi|\mathcal{F}_t]$ under $\mathbb{P}$, set $B\set
  Z/P$, and suppose that the $J$-dimensional martingale $S=(S^j_t)$
  under $\mathbb{Q}$ with the terminal value~\eqref{eq:1} is
  well-defined. If the $(B,S)$-market is complete, then
  $((B,S),(\widehat C^m)_{m=1,\dots,M})$ is a complete Radner
  equilibrium.
\end{Lemma}

The proofs are direct consequences of the identities
\begin{displaymath}
  \mathbb{E}^{\mathbb{Q}}[\int_0^1 \frac{dC_t}{B_t}] =
  \mathbb{E}[Z_1\int_0^1 \frac{dC_t}{B_t}] = 
  \mathbb{E}[\int_0^1 \frac{Z_t}{B_t} dC_t],
\end{displaymath}
which hold for an optional process $B>0$, an optional non-decreasing
process $C$, and a probability measure $\mathbb{Q}\sim \mathbb{P}$
with the density process $Z$.

\section{Model's primitives}
\label{sec:model-primitives}

The primitives of our model are defined in terms of the
$d$-dimensional diffusion
\begin{equation}
  \label{eq:3}
  X_t = X_0+\int_0^t b(s,X_s)ds + \int_0^t\sigma(s,X_s) dW_s, \; t\in [0,1]. 
\end{equation}
Here $X_0\in \mathbb{R}^d$, $W$ is a Brownian motion with values in
$\mathbb{R}^d$ defined on a complete filtered probability space
$(\Omega, \mathcal{F}_1, \mathbf{F} = (\mathcal{F}_t)_{t\in [0,1]},
\mathbb{P})$, and the drift
$\map{b=b(t,x)}{[0,1]\times\mathbb{R}^d}{\mathbb{R}^d}$ and the
volatility $\map{\sigma=\sigma(t,x)}{[0,1]\times
  \mathbb{R}^d}{\mathbb{R}^{d\times d}}$ are such that for all
$i,j=1,\dots,d$:
\begin{enumerate}[label=(A\arabic{*}), ref=(A\arabic{*})]
\item \label{item:1} the maps $t\mapsto b^i(t,\cdot)$ of $[0,1]$ to
  $\mathbf{L}_{\infty}= \mathbf{L}_{\infty}(\mathbb{R}^d)$ and
  $t\mapsto \sigma^{ij}(t,\cdot)$ of $[0,1]$ to $\mathbf{C} =
  \mathbf{C}(\mathbb{R}^d)$ are analytic on $(0,1)$ and H\"older
  continuous on $[0,1]$.  For $t\in [0,1]$ and $x\in \mathbb{R}^d$ the
  matrix $\sigma(t,x)$ has the inverse $\sigma^{-1}(t,x)$ and there
  exists a constant $N>0$, same for all $t$ and $x$, such that
  \begin{equation}
    \label{eq:4}
    \abs{\sigma^{-1}(t,x)} \leq N.
  \end{equation} 
  Moreover, there exists a strictly increasing function $\omega =
  (\omega(\epsilon))_{\epsilon>0}$ such that $\omega(\epsilon)\to 0$
  as $\epsilon\downarrow 0$ and, for all $t\in [0,1]$ and all $x,y\in
  \mathbb{R}^d$,
  \begin{displaymath}
    \abs{\sigma(t,x) - \sigma(t,y)} \leq \omega(\abs{x-y}).
  \end{displaymath}
  \setcounter{item}{\value{enumi}}
\end{enumerate}
In view of~\eqref{eq:4}, we can assume that the filtration
$\mathbf{F}$ is generated by $X$:
\begin{equation}
  \label{eq:5}
  \mathbf{F} = \mathbf{F}^X \set (\mathcal{F}^X_t)_{t\in [0,1]},
\end{equation}
where, as usual, $\mathcal{F}_t^X$ denotes the $\sigma$-field
generated by $(X_s)_{s\leq t}$ and complemented with $\mathbb{P}$-null
sets.

\begin{Remark}
  \label{rem:2}
  With respect to $x$, the conditions in \ref{item:1} are,
  essentially, the minimal classical assumptions guaranteeing the
  existence and the uniqueness of the weak solution to~\eqref{eq:3};
  see \citet[Theorem~7.2.1]{StrVarad:06} and
  \citet{Kryl:69,Kryl:72}. This weak solution is also well-defined
  when $b$ and $\sigma$ are only measurable functions with respect to
  $t$.  Example~2.5 in \cite{KramPred:14} shows that the requirement
  on $\sigma=\sigma(t,x)$ to be $t$-analytic is, however, essential
  for our main Theorem~\ref{th:1} to hold.
\end{Remark}

\begin{Remark}
  \label{rem:3}
  Let us compare our assumptions on the diffusion $X$ with those in
  the literature.  In the pioneering paper \cite{AnderRaim:08}, $X$ is
  a Brownian motion.  In \cite{HugMalTrub:12} the conditions are
  imposed on the diffusion coefficients $b=b(t,x)$ and $\sigma =
  \sigma(t,x)$ and on the transition density $p=p(t,x,s,y)$. In the
  main body of~\cite{HugMalTrub:12}, it is assumed that $b$, $\sigma$,
  and $p$ are analytic functions with respect to all their arguments.
  In the technical appendix to \cite{HugMalTrub:12}, these functions
  are required to be analytic with respect to $t$ and $s$ and
  $5$-times ($7$-times for $p$) continuously differentiable with
  respect to $x$ and $y$.  In \cite{RiedHerz:13} the diffusion
  coefficients $b$ and $\sigma$ do not depend on $t$, the matrix
  $\sigma$ is invertible and $b=b(x)$, $\sigma=\sigma(x)$, and
  $\sigma^{-1} = \sigma^{-1}(x)$ are bounded and analytic functions.

  From the point of view of applications, the most severe constraint
  of our setup is the boundedness assumption on the diffusion
  coefficients. This condition was used in the backward martingale
  representation theorem in~\cite{KramPred:14}, on which this paper
  relies, to facilitate references to the results from elliptic PDEs.
\end{Remark}

\subsection{Notional and dividends}
\label{sec:stocks-dividends}

From now on, the uncertainty and the information flow are modeled by
the filtered probability space $(\Omega, \mathcal{F}_1, \mathbf{F} =
(\mathcal{F}_t)_{t\in [0,1]}, \mathbb{P})$ with the filtration
$\mathbf{F}$ generated by the diffusion $X$ from~\eqref{eq:3}.

We assume that the notional $\Psi$, for the num\'eraire, and the
intermediate dividend rates $\theta=(\theta^j_t)$ and the terminal
dividends $\Theta = (\Theta^j)$, for the stocks, have the form:
\begin{align*}
  \Psi & = G(X_1) e^{\int_0^1 q(s,X_s)ds}, \\
  \theta^j_t &= f^j(t,X_t) e^{\int_0^t p^j(s,X_s)ds}, \; t\in [0,1],\\
  \Theta^j &= G(X_1)F^j(X_1)e^{\int_0^1 p^j(s,X_s)ds},\; j=1,\dots,J,
\end{align*}
where the functions $G,F^j$ on $\mathbb{R}^d$ and $q,f^j,p^j$ on
$[0,1]\times \mathbb{R}^d$ satisfy
\begin{enumerate}[label=(A\arabic{*}), ref=(A\arabic{*})]
  \setcounter{enumi}{\value{item}}
\item \label{item:2} The functions $G$ and $F^j$ are weakly
  differentiable, $G>0$, the Jacobian matrix $\left(F^j_{x_i} \set
    \frac{\partial F^j}{\partial x_i}\right)_{j=1,\dots,J,\;
    i=1,\dots,d}$ has rank $d$ almost surely under the Lebesgue
  measure on $\mathbb{R}^d$, and there is a constant $N>0$ such that
  \begin{displaymath}
    \abs{G_{x_i}(x)} +  \abs{F^j_{x_i}(x)} \leq e^{N(1+\abs{x})}, 
    \; x\in \mathbb{R}^d,  
  \end{displaymath}
  and such that $t\mapsto e^{-N \abs{\cdot}} f^j(t,\cdot) \set
  \bigl(e^{-N \abs{x}} f^j(t,x)\bigr)_{x\in \mathbf{R}^d}$ and
  $t\mapsto q(t,\cdot)$, $t\mapsto p^j(t,\cdot)$ are H\"older
  continuous maps of $[0,1]$ to $\mathbf{L}_{\infty}$ whose
  restrictions on $(0,1)$ are analytic.
  \setcounter{item}{\value{enumi}}
\end{enumerate}
The expressions for $\Psi$, $\theta^j$, and $\Theta^j$ are similar to
those in~\cite{AnderRaim:08}, where the rate functions $q$ and $p^j$
equal to zero.  In \cite{HugMalTrub:12} and \cite{RiedHerz:13} the
dividend's rate functions $f^j$ are time-independent: $f^j(t,x) =
f^j(x)$.

\begin{Remark}
  \label{rem:4}
  When the diffusion coefficients $\sigma^{ij}$ and $b^i$ and the
  functions $f^j$, $q$, and $p^j$ are also $x$-analytic it is enough
  to assume that the Jacobian matrix $\left(F^j_{x_i}\right)$ has rank
  $d$ only on an open set (equivalently, at just one point), see
  \cite{AnderRaim:08}, \cite{HugMalTrub:12}, and
  \cite{RiedHerz:13}. Without $x$-analyticity this is not possible,
  see Example~2.7 in \cite{KramPred:14}.
\end{Remark}

\begin{Remark}
  \label{rem:5}
  The $t$-analyticity condition on $f^j$ cannot be omitted; see
  Example~2.6 in \cite{KramPred:14} and the technical appendix
  to~\cite{HugMalTrub:12}. We stress that $t\mapsto
  e^{-N\abs{\cdot}}f^j(t,\cdot)$ is analytic as a \emph{map} of
  $(0,1)$ to $\mathbf{L}_{\infty}$; same is true for the rate
  functions $p^j$ and $q$ and the diffusion coefficients $b^i$ and
  $\sigma^{ij}$. This is more than just boundedness and the
  analyticity of $t\mapsto e^{-N\abs{x}}f^j(t,x)$ for every $x\in
  \mathbb{R}^d$. For instance, the map $t\mapsto (e^{-N\abs{x}}
  \sin(te^{x^2}))_{x\in\mathbb{R}^d}$ of $(0,1)$ to
  $\mathbf{L}_\infty$ is not even differentiable. The use of maps is
  essential in the proof of the backward martingale representation
  theorem from~\cite{KramPred:14} based on the theory of analytic
  semigroups. This result plays a key role in our study.
\end{Remark}

\subsection{Preferences and endowments}
\label{sec:preferences}

The agents consume continuously on $[0,1]$ according to an optional
process $\xi=(\xi_t)\geq 0$ of consumption rates and also at maturity
$t=1$ according to a random variable $\Xi\geq 0$ of terminal
wealth. The process of cumulative consumption is thus given by
\begin{displaymath}
  C_t = \int_0^t \xi_s ds + \Xi \ind{t=1}, \; t\in [0,1]. 
\end{displaymath}
The expected utility of $m$th agent has the form:
\begin{displaymath}
  \mathbb{U}^m(C) \set \mathbb{E}[\int_0^1 u^m(t,\xi_t,X_t)  e^{-\int_0^t
    r(s,X_s)ds}dt + U^m(\Xi,X_1) e^{-\int_0^1 r(t,X_t)dt}],
\end{displaymath}
where $r=r(t,x)$ is the ``impatience'' rate, common among the agents,
and $u^m=u^m(t,c,x)$ and $U^m=U^m(c,x)$ are utility functions for
intermediate and terminal consumptions defined for $t\in [0,1]$,
$c\geq 0$, and $x\in \mathbb{R}^d$. These expressions are similar to
those in \cite{AnderRaim:08}, where the impatience rate $r=r(t,x)$
does not depend on $x$.

The income process of $m$th agent is given by
\begin{displaymath}
  I^m_t = \int_0^t \lambda^m_s ds + \Lambda^m \ind{t=1}, \; t\in [0,1],  
\end{displaymath}
where the optional process $\lambda^m$ of income rates and the random
variable $\Lambda^m$ of terminal endowment satisfy
\begin{equation}
  \label{eq:6}
  \lambda^m\geq 0, \; 
  \Lambda^m\geq 0, \; \text{and} \;  \mathbb{P}[I^m_1>0]>0, \; m=1,\dots,M. 
\end{equation}
The total terminal and intermediate incomes are denoted by
\begin{displaymath}
  \Lambda
  \set \sum_{m=1}^M\Lambda^m\quad\text{and}\quad\lambda \set
  \sum_{m=1}^M \lambda^m. 
\end{displaymath}

We shall say that a function $f=f(c)$ on $[0,\infty)$ satisfies
\emph{the Inada conditions} if $f$ is strictly concave, strictly
increasing, and continuously differentiable on $(0,\infty)$ and
$\lim_{c\downarrow 0}f_c(c) = \infty$, $\lim_{c\to \infty}f_c(c) = 0$.
Moreover, $f(0) = \lim_{c\downarrow 0}f(c)$; this limit may equal
$-\infty$.

We impose the following conditions on $r$, $(U^m)$, and $\Lambda$:
\begin{enumerate}[label=(A\arabic{*}), ref=(A\arabic{*})]
  \setcounter{enumi}{\value{item}}
\item \label{item:3} $t\mapsto r(t,\cdot)$ is a H\"older continuous
  map of $[0,1]$ to $\mathbf{L}_{\infty}$ whose restriction on $(0,1)$
  is analytic.
\item \label{item:4} For $x\in \mathbb{R}^d$, the terminal wealth
  utility function $U^m(\cdot,x)$ on $[0,\infty)$ satisfies the Inada
  conditions. On $(0,\infty)\times \mathbb{R}^d$ the derivatives
  $U^m_{cc}$ and $U^m_{cx^i}$ exist and are continuous, $U^m_{cc}<0$,
  and, for some constant $N>0$,
  \begin{align}
    \label{eq:7}
    \abs{U^m(1,x)} &\leq e^{N(1+\abs{x})}, \; x\in \mathbb{R}^d, \\
    \label{eq:8}
    \left(-\frac{cU^m_{cc}}{U^m_c} +
      \frac{\abs{U^m_{cx^i}}}{U^m_c}\right)(c,x) &\leq N, \; (c,x)\in
    (0,\infty)\times \mathbb{R}^d.
  \end{align}
\item \label{item:5} $\Lambda = e^{H(X_1)}$, where the function
  $H=H(x)$ is weakly differentiable, and, for some constant $N>0$,
  \begin{displaymath}
    \abs{H(x)} \leq N(1+\abs{x}) \quad\text{and}\quad\abs{H_{x^i}}\leq
    e^{N(1+\abs{x})}, \; x\in \mathbb{R}^d.
  \end{displaymath}
  \setcounter{item}{\value{enumi}}
\end{enumerate}

\begin{Remark}
  \label{rem:6}
  In the state-homogeneous case, where $U^m(c,x) = U^m(c)$,
  inequality~\eqref{eq:7} holds trivially, while \eqref{eq:8} means
  the boundedness of the risk-aversion coefficient $-cU^m_{cc}/U^m_c$.
  Theorem~\ref{th:2} below shows that the families of functions
  $U=U(c,x)$ satisfying~\ref{item:4} are convex cones closed under
  $\sup$-convolution with respect to $c$.
\end{Remark}

The assumptions on the utility functions $u^m = u^m(t,c,x)$ for
intermediate consumption are bundled with the conditions on the total
income rate $\lambda$.  We assume that
\begin{displaymath}
  \lambda_t = e^{h(t,X_t)}, \; t\in [0,1], 
\end{displaymath}
and that the functions $u^m$ and $h$ are either time homogeneous:
\begin{enumerate}[label=(A\arabic{*}), ref=(A\arabic{*})]
  \setcounter{enumi}{\value{item}}
\item \label{item:6} $u^m(t,c,x) = u^m(c,x)$. For $x\in \mathbb{R}^d$
  the function $u^m(\cdot,x)$ on $[0,\infty)$ satisfies the Inada
  conditions. There is a constant $N>0$ such that
  \begin{equation}
    \label{eq:9}
    \abs{u^m(e^y,x)} \leq e^{N(1+\abs{x}+\abs{y})}, \; (x,y)\in
    \mathbb{R}^d \times \mathbb{R}.  
  \end{equation}
\item \label{item:7} $h(t,x) = h(x)$ and has a linear growth: for some
  $N\geq 0$,
  \begin{displaymath}
    \abs{h(x)} \leq N(1+\abs{x}), \; x\in \mathbb{R}^d. 
  \end{displaymath}
  \setcounter{item}{\value{enumi}}
\end{enumerate}
or they satisfy
\begin{enumerate}[label=(A\arabic{*}), ref=(A\arabic{*})]
  \setcounter{enumi}{\value{item}}
\item \label{item:8} For $(t,x) \in [0,1]\times \mathbb{R}^d$ the
  function $u^m(t,\cdot,x)$ satisfies the Inada conditions. The
  derivatives $u^m_{ct}$ and $u^m_{cc}$ exist and $u^m_{cc}<0$. There
  is a constant $N>0$ such that
  \begin{equation}
    \label{eq:10}
    \abs{u^m(t,e^y,x)} \leq e^{N(1+\abs{x} + \abs{y})}, \; (t,x,y)\in
    [0,1]\times \mathbb{R}^d\times \mathbb{R},
  \end{equation}
  and there is an open set $V\subset (0,\infty)^2$ containing
  $(0,1)\times \braces{1}$ (a \emph{neighborhood} of $(0,1)\times
  \braces{1}$) such that
  \begin{displaymath}
    (t,s)\mapsto (g(t,se^y,x))_{(x,y)\in \mathbb{R}^d\times\mathbb{R}},
  \end{displaymath}
  is a bounded analytic map of $V$ to
  $\mathbf{L}_{\infty}(\mathbb{R}^{d+1})$, where $g=g(t,c,x)$ stands
  for ${u^m_{ct}}/{u^m_c}$, ${cu^m_{cc}}/{u^m_c}$, and
  ${u^m_{c}}/({cu^m_{cc}})$.
\item \label{item:9} $h(t,x) = h_1(t,x) + h_2(x)$, where $t\mapsto
  h_1(t,\cdot)$ is a H\"older continuous map of $[0,1]$ to
  $\mathbf{L}_\infty(\mathbb{R}^{d})$ whose restriction on $(0,1)$ is
  analytic and the function $h_2=h_2(x)$ has a linear growth: for some
  $N\geq 0$,
  \begin{displaymath}
    \abs{h_2(x)} \leq N(1+\abs{x}), \; x\in \mathbb{R}^d. 
  \end{displaymath}  \setcounter{item}{\value{enumi}}
\end{enumerate}
The role of either pair of these assumptions is to imply the
assertions of Lemma~\ref{lem:10} for every aggregate utility function
$u(w) = u(t,c,x;w)$ introduced in Section~\ref{sec:main-results}.

In \cite{HugMalTrub:12} and \cite{RiedHerz:13} the common impatience
rate $r=r(t,x)$ is constant, the utility functions $u^m=u^m(t,c,x)$
for intermediate consumption depend only on $c$: $u^m(t,c,x) =
u^m(c)$, and the total income rate function $h=h(t,x)$ is
time-homogeneous: $h(t,x) = h(x)$.

\begin{Remark}
  \label{rem:7}
  A classical example of a utility function in~\ref{item:8} is
  \begin{equation}
    \label{eq:11}
    u^m(t,c,x) \set e^{\nu^m(t)}  \frac{c^{1-a^m}- 1}{1-a^m} g^m(x), 
  \end{equation}
  where $a^m$ is a positive constant of risk-aversion, $\nu^m =
  \nu^m(t)$ is an analytic function on $(0,1)$ with bounded
  derivative, and $g^m=g^m(x)$ is a positive function with exponential
  growth: $0<g^m(x) \leq Ne^{N\abs{x}}$, for some $N>0$; if $a^m=1$,
  then, by continuity, $u^m(t,c,x) \set e^{\nu^m(t)} (\ln c) g^m(x)$.

  Theorem~\ref{th:2} shows that the families of functions $u=u(t,c,x)$
  satisfying either~\ref{item:6} or~\ref{item:8} are convex cones
  closed under $\sup$-convolution with respect to $c$. In particular,
  one can begin with functions as in~\eqref{eq:11} and build more
  general utility functions by successively taking positive linear
  combinations and applying $\sup$-convolutions with respect to $c$.

  Regarding~\ref{item:8} we also observe that for a function $g=
  g(t,c,x)$ the analyticity of the map $(t,s)\mapsto
  (g(t,se^y,x))_{(x,y)\in \mathbb{R}^d\times\mathbb{R}}$ of a
  neighborhood of $(0,1)\times \braces{1}$ to
  $\mathbf{L}_{\infty}(\mathbb{R}^{d+1})$ readily implies the
  analyticity of the map $(t,c)\to g(t,c,\cdot)$ of $(0,1)\times
  (0,\infty)$ to $\mathbf{L}_{\infty}(\mathbb{R}^d)$. The inverse is
  not true.  For example, the function
  \begin{displaymath}
    g(c) \set \sin (\ln^2(c)), \quad c>0,
  \end{displaymath}
  is analytic and uniformly bounded on $(0,\infty)$. However, the map
  $s\mapsto (g(se^y))_{y\in \mathbb{R}}$ taking values in
  $\mathbf{L}_{\infty}(\mathbb{R}^d)$ is not even continuous at $s =
  1$:
  \begin{align*}
    \limsup_{s\to 1} \sup_{y\in \mathbb{R}}\abs{g(se^{y}) - g(e^{y})}
    & = \limsup_{s\to 1} \sup_{y\in \mathbb{R}} \abs{\sin
      (\ln^2(se^{y})) -
      \sin(\ln^2(e^{y}))} \\
    &= \limsup_{\epsilon\to 0} \sup_{y\in \mathbb{R}} \abs{\sin
      ((\epsilon + y)^2) - \sin(y^2)} = 2.
  \end{align*}
\end{Remark}

\section{Main result}
\label{sec:main-results}

Denote by $\Sigma^M$ the simplex in $\mathbb{R}^M$:
\begin{displaymath}
  \Sigma^M \set \descr{w\in [0,\infty)^M}{\sum_{m=1}^M w^m = 1}.
\end{displaymath}
For a weight $w\in {\Sigma^M}$ define the aggregate utility functions
$U(w) = U(c,x;w)$ and $u(w) = u(t,c,x;w)$, where $(t,c,x)\in
[0,1]\times (0,\infty)\times \mathbb{R}^d$, as the $w$-weighted
$\sup$-convolutions with respect to $c$:
\begin{align*}
  U(c,x;w) &\set \sup\descr{\sum_{m=1}^M w^m U^m(c^m,x)}{c^m\geq 0, \;
    c^1+\dots + c^M = c}
  ,  \\
  u(t,c,x;w) &\set \sup\descr{\sum_{m=1}^M w^m u^m(t,c^m,x)}{c^m\geq
    0, \; c^1+\dots + c^M = c}.
\end{align*}
Theorem~\ref{th:2} shows that the aggregate utility functions $U(w)$
and $u(w)$ satisfy same conditions~\ref{item:4}, \ref{item:6},
and~\ref{item:8} as $U^m$ and $u^m$; in particular, they satisfy the
Inada conditions with respect to $c$.

By $(\Pi^m(w))_{m=1,\dots,M}$ we denote the $w$-weighted Pareto
allocation of $\Lambda$:
\begin{align*}
  \Pi^m(w) &\set 0 \quad \text{if} \quad w^m=0, \\
  w^m U_c^m(\Pi^m(w),X_1) &\set U_c(\Lambda,X_1;w) \quad \text{if}
  \quad w^m>0,
\end{align*}
and by $(\pi^m(w))_{m=1,\dots,M}$ the optional processes of Pareto
consumption rates:
\begin{align*}
  \pi^m(w) &\set 0 \quad \text{if} \quad w^m=0, \\
  w^m u^m_c(t, \pi_t^m(w),X_t) &\set u_c(t,\lambda_t,X_t;w), \; t\in
  [0,1], \quad \text{if} \quad w^m>0.
\end{align*}
The cumulative Pareto consumption processes are given by
\begin{equation}
  \label{eq:12}
  C^m_t(w) \set \int_0^t \pi^m_s(w) ds + \Pi^m(w) \ind{t=1}, \;
  t\in [0,1], \; m=1,\dots,M. 
\end{equation}

We denote by $\mathcal{W}$ the subset of $\Sigma^M$ such that
\begin{displaymath}
  \mathcal{W} \set \descr{w\in \Sigma^M}{\Phi^m(w) = 0, \; m=1,\dots,M},  
\end{displaymath}
where the function $\Phi^m$ on $\Sigma^M$ is given by
\begin{align*}
  \Phi^m(w) & \set \mathbb{E}\bigl[e^{-\int_0^1
    r(t,X_t)dt} U_c(\Lambda,X_1;w)(\Pi^m(w) - \Lambda^m) \\
  & \quad + \int_0^1 e^{-\int_0^t r(s,X_s)ds}
  u_c(t,\lambda_t,X_t;w)(\pi^m_t(w) - \lambda^m_t)dt\bigr].
\end{align*}
These functions are well-defined; see Lemma~\ref{lem:11}.

For $w\in \Sigma^M$ define the martingale $Y(w)$ by
\begin{displaymath}
  Y_t(w)  \set \mathbb{E}[\Psi U_c(\Lambda,X_1;w) e^{-\int_0^1
    r(s,X_s)ds}|\mathcal{F}_t], \; t\in [0,1], 
\end{displaymath}
and the probability measure $\mathbb{Q}(w)\sim \mathbb{P}$ by
\begin{displaymath}
  \frac{d\mathbb{Q}(w)}{d\mathbb{P}} \set \frac{Y_1(w)}{Y_0(w)}.  
\end{displaymath}
We also denote by $B(w)$ the positive optional process
\begin{displaymath}
  B_t(w) \set \frac{Y_t(w)}{u_c(t,\lambda_t,X_t;w)}e^{\int_0^t r(s,X_s)ds}
  \ind{t<1} + \Psi \ind{t=1}, \; t\in [0,1], 
\end{displaymath}
and by $S(w) = (S^j_t(w))$ the $J$-dimensional
$\mathbb{Q}(w)$-martingale with the terminal value
\begin{displaymath}
  S^j_1(w) \set \frac{\Theta^j}{\Psi} + \int_0^1
  \frac{\theta^j_u}{B_u(w)} du, \; j=1,\dots,J.
\end{displaymath}
The integrability conditions needed for the existence of such $Y(w)$
and $S(w)$ are verified in Lemmas~\ref{lem:12} and~\ref{lem:13}.

We now state the main result of the paper.

\begin{Theorem}
  \label{th:1}
  Let the conditions \eqref{eq:5}, \eqref{eq:6},
  \ref{item:1}--\ref{item:5}, and either \ref{item:6}--\ref{item:7} or
  \ref{item:8}--\ref{item:9} hold. Then there is a complete Radner
  equilibrium. The set $\mathcal{W}$ is not empty and belongs to the
  interior of $\Sigma^M$. Every complete Radner equilibrium has the
  form $((B(w),S(w)), C(w))$ for some $w\in \mathcal{W}$ and,
  conversely, $((B(w),S(w)), C(w))$ is a complete Radner equilibrium
  for every $w\in \mathcal{W}$.
\end{Theorem}

The proof is given in Section~\ref{sec:proof-th:1} and relies on the
criteria for the existence of Arrow-Debreu equilibria from
\cite{Kram:13}, on the backward martingale representation result from
\cite{KramPred:14}, and on the study of utility functions in
Section~\ref{sec:conv-funct} below.

\section{Convex cones of utility functions closed under
  $\sup$-convolution}
\label{sec:conv-funct}

Let $\mathcal{U}$ be a family of real-valued (utility) functions $u =
u(t,c,x)$ on $[0,1]\times (0,\infty)\times \mathbb{R}^d$ which are
concave with respect to $c$.  We are interested in $\mathcal{U}$ being
a convex cone closed under the operations of $\sup$-convolution with
respect to $c$: for every $u$, $u_1$, and $u_2$ in $\mathcal{U}$ and
every constant $a>0$ the functions $au$, $u_1 + u_2$, $(u_1 \oplus_c
u_2)$ belong to $\mathcal{U}$, where
\begin{equation}
  \label{eq:13}
  (u_1 \oplus_c u_2)(t,c,x) \set \sup\descr{u_1(t,c_1,x) +
    u_2(t,c_2,x)}{c_i> 0, \; c_1 + c_2 = c}.
\end{equation}

Motivated by the proof of Theorem~\ref{th:1} this property will be
established for the following families of functions:
\begin{description}
\item[$\mathcal{U}_1$] consists of measurable functions $u=u(t,c,x)$
  on $[0,1]\times (0,\infty)\times \mathbb{R}^d$ which satisfy the
  Inada conditions with respect to $c$ and such that, for some
  constant $N=N(u)>0$,
  \begin{equation}
    \label{eq:14}
    \abs{u(t,e^y,x)} \leq e^{N(1+\abs{x}+\abs{y})}, \; (t,x,y)\in
    [0,1]\times\mathbb{R}^d\times \mathbb{R}. 
  \end{equation}
\item[$\mathcal{U}_2$] consists of functions $u\in \mathcal{U}_1$ such
  that the derivatives $u_{cc}$ and $u_{cx^i}$ exist and are
  continuous functions with respect to $(c,x)$, $u_{cc}<0$, and, for
  some constant $N=N(u)>0$,
  \begin{equation}
    \label{eq:15}
    \left(-\frac{cu_{cc}}{u_c} +
      \frac{\abs{u_{cx^i}}}{u_c}\right)(t,c,x) \leq N, \;
    (t,c,x)\in [0,1]\times (0,\infty)\times \mathbb{R}^d. 
  \end{equation}
\item[$\mathcal{U}_3$] consists of functions $u\in \mathcal{U}_1$ such
  that the derivatives $u_{ct}$ and $u_{cc}$ exist, $u_{cc}<0$, and
  \begin{equation}
    \label{eq:16}
    (t,s)\mapsto  \left(g(t,se^y,x)\right)_{(x,y)\in 
      \mathbb{R}^d\times\mathbb{R}}
  \end{equation}
  is a bounded analytic map of some neighborhood of $(0,1)\times
  \braces{1}$ to $\mathbf{L}_{\infty}(\mathbb{R}^{d+1})$, where
  $g=g(t,c,x)$ stands for ${u_{ct}}/{u_c}$, ${cu_{cc}}/{u_c}$, and
  ${u_{c}}/(cu_{cc})$.
\end{description}

\begin{Theorem}
  \label{th:2}
  Each of the families $\mathcal{U}_i$, $i=1,2,3$, is a convex cone
  closed under $\sup$-convolution with respect to $c$.
\end{Theorem}

The rest of this section is devoted to the proof of this theorem which
we divide into lemmas. In the study of $\mathcal{U}_3$ we use the
versions of composition and implicit function theorems for analytic
maps with values in $\mathbf{L}_{\infty}$ stated in
Appendix~\ref{sec:analytic-maps}.

\begin{Lemma}
  \label{lem:3}
  Each of the families $\mathcal{U}_i$, $i=1,2,3$, is a convex cone.
\end{Lemma}

\begin{proof}
  For $\mathcal{U}_1$ and $\mathcal{U}_2$ the result is
  straightforward.

  Let $u^1$ and $u^2$ be in $\mathcal{U}_3$, denote $u\set u^1+u^2$,
  and fix $t_0\in (0,1)$. The analyticity of the map~\eqref{eq:16} at
  the point $(t_0,1)$, with $g=g(t,c,x)$ standing for ${u_{ct}}/{u_c}$
  and ${cu_{cc}}/{u_c}$, readily follows if we can show that
  \begin{equation}
    \label{eq:17}
    (t,s)\mapsto \left(w(t,s e^y,x) \set \frac{u^1_{c}}{u^1_{c} +
        u^2_{c}}(t,s e^y,x)\right)_{(x,y)\in \mathbb{R}^d\times \mathbb{R}}
  \end{equation}
  is an analytic map of some neighborhood of $(t_0,1)$ to
  $\mathbf{L}_{\infty}$.

  For $i=1,2$ denote $a_i \set -{cu^i_{cc}}/{u^i_{c}}$ and $q_i \set
  {u^i_{ct}}/{u^i_{c}}$. Observe that
  \begin{align*}
    f_i(t,sc,x) &\set \frac{u^i_{c}(t,sc,x)}{u^i_{c}(t_0,c,x)} =
    \exp(-\int_{1}^s \frac1r a_i(t,rc,x) dr + \int_{t_0}^t q_i(r, c,x)
    dr), \\
    w(t,sc,x) &= \frac{w(t_0,c, x) f_1(t,sc,x)}{w(t_0,c,x) f_1(t,sc,x)
      + (1-w(t_0,c,x)) f_2(t,sc,x)}.
  \end{align*}
  The analyticity of $a_i$ and $q_i$ in $\mathcal{U}_3$ and
  Theorem~\ref{th:3} yield the analyticity of the map $(t,s)\mapsto
  \left(f_i(t,s e^{y},x)\right)_{(x,y)\in \mathbb{R}^{d+1}}$ of some
  neighborhood of $(t_0,1)$ to $\mathbf{L}_{\infty}$. As
  $f_i(t_0,c,x)=1$ and
  \begin{displaymath}
    (p,q) \mapsto \left( \frac{w(t_0,e^y,x)p}{w(t_0,e^y,x)p +
        (1-w(t_0,e^y,x))q}\right)_{(x,y)\in \mathbb{R}^d\times 
      \mathbb{R}} 
  \end{displaymath}
  is an analytic map of $(0,\infty)^2$ to $\mathbf{L}_{\infty}$,
  Theorem~\ref{th:3} yields the required analyticity of the
  map~\eqref{eq:17}; hence, also the analyticity of the
  map~\eqref{eq:16} at $(t_0,1)$ with $g=g(t,c,x)$ standing for
  ${u_{ct}}/{u_c}$ and ${cu_{cc}}/{u_c}$.

  As $a\set-{cu_{cc}}/{u_c}$ is a convex combination of the
  corresponding risk-aversions for $u^1$ and $u^2$, there is $N>0$
  such that $\frac1N\leq a\leq N$.  Another application of
  Theorem~\ref{th:3} implies the analyticity of the map~\eqref{eq:16}
  at $(t_0,1)$ with $g\set 1/a$.
\end{proof}

\begin{Lemma}
  \label{lem:4}
  Let $u^1,u^2\in \mathcal{U}_1$. Then $u\set u^1\oplus_c u^2$ belongs
  to $\mathcal{U}_1$ and the upper bound in~\eqref{eq:13} is attained
  on $\widehat c_1 = f$ and $\widehat c_2 = c - f$, where the function
  $f=f(t,c,x)$ is such that $0<f<c$ and
  \begin{equation}
    \label{eq:18}
    u_c(t,c,x) = u^1_c(t,f(t,c,x),x) = u^2_c(t,c-f(t,c,x),x).  
  \end{equation}
\end{Lemma}

\begin{proof}
  Elementary arguments show that $u(t,\cdot,x)$ satisfies the Inada
  conditions and that the upper bound in~\eqref{eq:13} is attained on
  the above $\widehat c_1$ and $\widehat c_2$. Since
  \begin{displaymath}
    (u^1 + u^2)(t,c/2,x) \leq u(t,c,x) \leq (u^1 + u^2)(t,c,x)
  \end{displaymath}
  we deduce that $u$ satisfies~\eqref{eq:14}. Hence, $u$ belongs to
  $\mathcal{U}_1$.
\end{proof}

\begin{Lemma}
  \label{lem:5}
  Let $u^1,u^2\in \mathcal{U}_2$. Then $u\set u^1\oplus_c u^2$ belongs
  to $\mathcal{U}_2$ and
  \begin{align}
    \label{eq:19}
    \frac{u_c}{u_{cc}}(t,c,x) &= \frac{u^1_c}{u^1_{cc}}(t,f,x) +
    \frac{u^2_c}{u^2_{cc}}(t,c-f,x), \\
    \label{eq:20}
    \frac{u_{cx^i}}{u_{cc}}(t,c,x) &=
    \frac{u^1_{cx^i}}{u^1_{cc}}(t,f,x) +
    \frac{u^2_{cx^i}}{u^2_{cc}}(t,c-f,x),
  \end{align}
  where $f=f(t,c,x)$ is given by~\eqref{eq:18}.
\end{Lemma}

\begin{proof}
  By Lemma~\ref{lem:4}, $u\in \mathcal{U}_1$. From~\eqref{eq:18} and
  the implicit function theorem we deduce that the derivatives $f_c$
  and $f_{x^i}$ and, then, $u_{cc}$ and $u_{cx^i}$ exist and are
  continuous functions with respect to $(c,x)$.  Direct computations
  show that $u_{cc}<0$ and that the identities~\eqref{eq:19} and
  \eqref{eq:20} hold.  It follows that ${u_c}/({cu_{cc}})$ and
  ${u_{cx^i}}/{u_c}$ are convex combinations of the corresponding
  coefficients of $u^1$ and $u^2$.  This implies~\eqref{eq:15}. Hence,
  $u\in \mathcal{U}_2$.
\end{proof}

The verification of the closure of $\mathcal{U}_3$ under $\oplus_c$
relies on

\begin{Lemma}
  \label{lem:6}
  Let $u^1,u^2\in \mathcal{U}_3$, $u\set u^1\oplus_c u^2$, and the
  function $f=f(t,c,x)$ on $[0,1]\times (0,\infty)\times \mathbb{R}^d$
  be given by~\eqref{eq:18}.  Let $t_0 \in (0,1)$. Then
  \begin{displaymath}
    (t,s)\mapsto  
    \left(\frac{f(t,se^y,x)}{f(t_0,e^y,x)}, \;
      \frac{se^y-f(t,se^y,x)}{e^y-f(t_0,e^y,x)}\right)_{(x,y)\in  
      \mathbb{R}^d\times\mathbb{R}}
  \end{displaymath}
  are analytic maps of a neighborhood of the point $(t_0,1)$ to
  $\mathbf{L}_{\infty}(\mathbb{R}^{d+1})$.
\end{Lemma}

\begin{proof}
  For $t\in [0,1]$, $s>0$, $r>0$, $y\in \mathbb{R}$, and $x\in
  \mathbb{R}^d$ denote
  \begin{align*}
    h_1(t,s,y,x) &\set \frac{f(t,se^y,x)}{se^y-f(t,se^y,x)},
    \\
    h_2(t,s,y,x) &\set \frac{h_1(t,s,y,x)}{h_1(t_0,1,y,x)}, \\
    h_3(r,y,x) &\set \frac{1+h_1(t_0,1,y,x)}{1+rh_1(t_0,1,y,x)}
  \end{align*}
  and observe that $h_i>0$ and
  \begin{equation}
    \label{eq:21}
    \begin{split}      
      f(t,se^y,x) &= se^y \frac{h_1}{1+h_1} = sf(t_0,e^y,x)h_2
      h_3(h_2,y,x), \\
      se^y - f(t,se^y,x) &= se^y \frac{1}{1+h_1} = s(e^y -
      f(t_0,e^y,x)) h_3(h_2,y,x),
    \end{split}
  \end{equation}
  where we omitted the argument $(t,s,y,x)$ for $h_1$ and $h_2$. As
  $h_2(t_0,1, \cdot,\cdot) = 1$ and $r\mapsto h_3(r,\cdot,\cdot)$ is
  an analytic map of $(0,\infty)$ to
  $\mathbf{L}_{\infty}(\mathbb{R}^{d+1})$, the result follows
  from~\eqref{eq:21} and Theorem~\ref{th:3} in the appendix if
  \begin{equation}
    \label{eq:22}
    (t,s)\mapsto  \left(
      h_2(t,s,y,x)\right)_{(x,y)\in 
      \mathbb{R}^d\times\mathbb{R}}
  \end{equation}
  is an analytic map of a neighborhood of the point $(t_0,1)$ to
  $\mathbf{L}_{\infty}(\mathbb{R}^{d+1})$.

  We rely on the version of implicit function theorem stated in
  Theorem~\ref{th:4}. Define the function
  \begin{equation}
    \label{eq:23}
    h_4(t,s,r,y,x) \set \ln
    \left(\frac{u^1_c(t,sf(t_0,e^y,x)rh_3,x)}{u^2_c(t,s(e^y-f(t_0,e^y,x))  
        h_3,x)}\right),
  \end{equation}
  where $h_3=h_3(r,y,x)$.  From~\eqref{eq:21} and~\eqref{eq:18} we
  deduce that
  \begin{displaymath}
    h_4(t,s,h_2,y,x) =  \ln
    \left(\frac{u^1_c(t,f(t,se^y,x),x)}{u^2_c(t,se^y-f(t,se^y,x),x)}\right)
    = 0.  
  \end{displaymath}
  Denote $a_i \set -cu^i_{cc}/u^c$ and $q_i \set u^i_{ct}/u^i_c$.
  Direct computations show that
  \begin{align*}
    \frac{\partial h_4}{\partial t} &= q_1 - q_2, \\
    \frac{\partial h_4}{\partial s} &= \frac1s(-a_1 + a_2), \\
    \frac{\partial h_4}{\partial r} &= -a_1 \frac{\partial}{\partial
      r}\ln(rh_3) + a_2 \frac{\partial}{\partial r}\ln h_3 \\
    &= -\frac1r \left(a_1 \frac1{1+rh_1} + a_2 \frac{rh_1}{1+rh_1}
    \right),
  \end{align*}
  where $h_1$ is evaluated at $(t_0,1,y,x)$ and the omitted arguments
  for $a_i$ and $q_i$ are as for $u^i_c$ in~\eqref{eq:23}. Another
  application of Theorem~\ref{th:3} yields that each of these partial
  derivatives for $h_4$ defines an analytic map of some neighborhood
  of $(t_0,1,1)$ to $\mathbf{L}_{\infty}$.  As $h_4(t_0,1,1,y,x) = 0$
  we obtain that $(t,s,r)\mapsto h_4(t,s,r,\cdot,\cdot)$ is also an
  analytic map of a neighborhood of $(t_0,1,1)$ to
  $\mathbf{L}_{\infty}$. Moreover, \ref{item:8} yields the existence
  of a constant $N>0$ such that $1/N \leq a_i \leq N$. Hence,
  \begin{displaymath}
    \frac1N \leq -\frac{\partial h_4}{\partial r}(t_0,1,1,\cdot,\cdot) \leq N. 
  \end{displaymath}
  The required analyticity of the map~\eqref{eq:22} follows now from
  Theorem~\ref{th:4}.
\end{proof}

The following lemma completes the proof of the theorem.

\begin{Lemma}
  \label{lem:7}
  Let $u^1,u^2\in \mathcal{U}_3$. Then $u\set u^1\oplus_c u^2$ belongs
  to $\mathcal{U}_3$.
\end{Lemma}

\begin{proof}
  Let $f=f(t,c,x)$ be defined by~\eqref{eq:18}.  From~\eqref{eq:18}
  and the implicit function theorem we obtain that the derivatives
  $f_t$ and $u_{ct}$ exist and
  \begin{equation}
    \label{eq:24}
    \frac{u_{ct}}{u_{cc}}(t,c,x) =
    \frac{u^1_{ct}}{u^1_{cc}}(t,f,x) +
    \frac{u^2_{ct}}{u^2_{cc}}(t,c-f,x).
  \end{equation}

  Denote $a \set -{cu_{cc}}/{u_c}$ and $q \set {u_{ct}}/{u_c}$ and let
  $a_i$ and $q_i$ be the corresponding coefficients for $u^i$,
  $i=1,2$.  From~\eqref{eq:19} and~\eqref{eq:24} we deduce that
  \begin{align*}
    \frac1a &= \frac1{a_1(t,f,x)}\frac{f}{c} +
    \frac1{a_2(t,c-f,x)}\frac{c-f}{c}, \\
    \frac{q}a &= \frac{q_1}{a_1}(t,f,x)\frac{f}{c} +
    \frac{q_2}{a_2}(t,c-f,x)\frac{c-f}{c}.
  \end{align*}
  The required boundedness and analyticity of the maps
  \begin{displaymath}
    (t,s) \mapsto (g(t,se^y,x))_{(x,y)\in \mathbb{R}^d\times
      \mathbb{R}}, \quad \text{$g$ stands for $a$,$1/a$, and $q$},  
  \end{displaymath}
  of some neighborhood of $(0,1)\times\braces{1}$ to
  $\mathbf{L}_{\infty}$ follows now from the boundedness and
  analyticity of these maps for $u^1$ and $u^2$ and from
  Lemma~\ref{lem:6} and Theorem~\ref{th:3}.
\end{proof}

\section{Proof of Theorem~\ref{th:1}}
\label{sec:proof-th:1}

As in \cite{AnderRaim:08}, \cite{HugMalTrub:12},
and~\cite{RiedHerz:13} the proof follows the road map outlined in
Section~\ref{sec:introduction} and consists of two steps. First, in
Lemma~\ref{lem:11}, we describe all Arrow-Debreu equilibria. Then, in
Lemma~\ref{lem:13}, we obtain that every Arrow-Debreu equilibrium
yields a complete Radner equilibrium.

We always work under the assumptions of Theorem~\ref{th:1}. We shall
often use the fact that, as the diffusion coefficients $b=b(t,x)$ and
$\sigma = \sigma(t,x)$ are bounded, the running maximum $\sup_{t\in
  [0,1]} \abs{X_t}$ has all exponential moments.

We begin with some estimates concerning the utility functions.

\begin{Lemma}
  \label{lem:8}
  Let $U=U(c,x)$ be a function on $(0,\infty) \times \mathbb{R}^d$
  satisfying~\ref{item:4}. Then there is a constant $N>0$ such that,
  for $(x,y)\in \mathbb{R}^d\times \mathbb{R}$,
  \begin{equation}
    \label{eq:25}
    (\abs{U} + U_c - U_{cc} + \abs{U_{cx^i}})(e^y,x)  \leq
    e^{N(1+\abs{x}+\abs{y})}. 
  \end{equation}
\end{Lemma}
\begin{proof}
  From~\ref{item:4} we deduce that the function $G(x,y) \set \ln
  U_c(e^y,x)$ has bounded derivatives. This yields the estimate for
  $U_c$. The inequalities for $U_{cc}$ and $U_{cx^i}$ then follow
  from~\eqref{eq:8}, while the inequality for $U$ follows
  from~\eqref{eq:7}.
\end{proof}

\begin{Lemma}
  \label{lem:9}
  Let $u=u(c,x)$ be a function on $(0,\infty) \times \mathbb{R}^d$
  satisfying \ref{item:6}. Then there is a constant $N>0$ such that
  \begin{displaymath}
    u_c(e^{y},x) \leq e^{N(1+\abs{x}+\abs{y})}, \; (x,y)\in
    \mathbb{R}^d\times \mathbb{R}. 
  \end{displaymath}
\end{Lemma}
\begin{proof}
  From the concavity of $u(\cdot,x)$ we deduce that
  \begin{displaymath}
    u_c(c,x) \leq \frac2c\braces{u(c,x) - u({c}/{2},x)} \leq
    \frac2c\braces{\abs{u(c,x)} +\abs{u({c}/{2},x)}}
  \end{displaymath}
  and the result follows from~\eqref{eq:9}.
\end{proof}

\begin{Lemma}
  \label{lem:10}
  Let $u=u(t,c,x)$ be a function on $[0,1]\times (0,\infty) \times
  \mathbb{R}^d$ satisfying \ref{item:8} and such that for every
  $(c,x)\in (0,\infty)\times \mathbb{R}^d$ the function
  $u_c(\cdot,c,x)$ on $[0,1]$ is continuous. Let $h=h(t,x)$ be a
  function on $[0,1]\times \mathbb{R}^d$ satisfying \ref{item:9}. Then
  there is a constant $N>0$ such that
  \begin{equation}
    \label{eq:26}
    t\mapsto e^{-N\abs{\cdot}}u_c(t,e^{h(t,\cdot)},\cdot) \set
    (e^{-N\abs{x}} u_c(t,e^{h(t,x)},x))_{x\in
      \mathbb{R}^d}
  \end{equation}
  is a H\"older continuous map of $[0,1]$ to $\mathbf{L}_{\infty}$
  whose restriction on $(0,1)$ is analytic.
\end{Lemma}

\begin{proof}
  From the concavity of $u(t,\cdot,x)$ we obtain
  \begin{displaymath}
    u_c(t,c,x) \leq \frac2{c}(u(t,c,x) - u(t,c/2,x))
  \end{displaymath}
  and then~\eqref{eq:10} implies the existence of $N>0$ such that
  \begin{equation}
    \label{eq:27}
    u_c(t,e^y,x) \leq e^{N(1+\abs{x}+\abs{y})}, \; (t,x,y)\in
    [0,1]\times\mathbb{R}^{d}\times\mathbb{R}. 
  \end{equation}

  Fix $t_0\in (0,1)$. Denote $a \set -{cu_{cc}}/{u_{c}}$ and $q \set
  {u_{ct}}/{u_{c}}$ and observe that
  \begin{displaymath}
    \frac{u_{c}(t,sc,x)}{u_{c}(t_0,c,x)} =
    \exp\left(-\int_{1}^s \frac1r a(t,rc,x) dr + \int_{t_0}^t q(r, c,x)
      dr\right).
  \end{displaymath}
  The analyticity of $a$ and $q$ in~\ref{item:8}, Theorem~\ref{th:3},
  and the inequality~\eqref{eq:27} yield the existence of $N>0$ such
  that
  \begin{equation}
    \label{eq:28}
    (t,s) \mapsto \left(e^{-N(\abs{x}+\abs{y})}
      u_c(t,se^y,x)\right)_{(x,y)\in \mathbb{R}^d \times \mathbb{R}} 
  \end{equation}
  is an analytic map of a neighborhood of $(0,1)\times\braces{1}$ to
  $\mathbf{L}_{\infty}(\mathbb{R}^{d+1})$. Recall that in~\ref{item:9}
  we have $h(t,x) = h_1(t,x) + h_2(x)$. Define the functions
  \begin{align*}
    g_1(t,s,x) &\set u_c\bigl(t,s e^{h(t_0,x)},x\bigr), \\
    g_2(t,x) &\set e^{h(t,x)-h(t_0,x)} = e^{h_1(t,x)-h_1(t_0,x)}.
  \end{align*}
  From the analyticity of~\eqref{eq:28} and the linear growth for
  $h(t_0,\cdot)$ we deduce the existence of $N>0$ such that $(t,s)
  \mapsto e^{-N\abs{\cdot}} g_1(t,s,\cdot)$ is an analytic map of a
  neighborhood of $(t_0,1)$ to $\mathbf{L}_{\infty}$.  We also obtain
  that $g_2>0$, $g_2(t_0,\cdot) = 1$, and, by~\ref{item:9} and
  Theorem~\ref{th:3}, the map $t\mapsto g_2(t,\cdot)$ of $(0,1)$ to
  $\mathbf{L}_\infty$ is analytic. As
  \begin{displaymath}
    u_c(t,e^{h(t,x)},x) = g_1(t,g_2(t,x),x),
  \end{displaymath}
  Theorem~\ref{th:3} now implies the analyticity of the
  map~\eqref{eq:26} at $t_0$; hence, also the analyticity of this map
  on $(0,1)$.

  It remains to verify the H\"older continuity of the
  map~\eqref{eq:26}.  Observe first that as the function
  $u_c(\cdot,c,x)$ on $[0,1]$ is continuous and the function
  $u_c(t,\cdot,x)$ on $(0,\infty)$ is continuous and decreasing, the
  map $t\mapsto u_c(t,\cdot,x)$, $t\in [0,1]$, is uniformly continuous
  on compact sets: for every $n>0$
  \begin{displaymath}
    \sup_{\abs{y}\leq n} \abs{u_c(t,e^y,x) -
      u_c(s,e^y,x)} \to 0 \quad \text{if} \quad \abs{s-t}\to 0, \;
    s,t\in [0,1]. 
  \end{displaymath}
  This property and the continuity of the function $h(\cdot,x)$ on
  $[0,1]$ implies the continuity of the function
  $u_c(\cdot,e^{h(\cdot,x)},x)$ on $[0,1]$ for every $x\in
  \mathbb{R}^d$.

  Fix $0<s<t<1$ and denote
  \begin{displaymath}
    f(r) = u_c(s+r(t-s),e^{h(s,x) + r(h(t,x)-h(s,x))},x), \; r\in
    [0,1]. 
  \end{displaymath}
  From~\eqref{eq:27}, \ref{item:8}, and~\ref{item:9} we deduce the
  existence of constants $M,N>0$ and $0<\delta<1$ such that, for $r\in
  (0,1)$,
  \begin{align*}
    f(r) & \leq e^{N(1+\abs{x})}, \\
    \abs{f'(r)} & \leq M \left((t-s) + \abs{h(t,x)-h(s,x)}\right) f(r)
    \leq e^{N(1+\abs{x})} \abs{t-s}^\delta.
  \end{align*}
  It follows that
  \begin{displaymath}
    \abs{u_c(t,e^{h(t,x)},x) - u_c(s,e^{h(s,x)},x)} = \abs{f(1)-f(0)}
    \leq e^{N(1+\abs{x})} \abs{t-s}^\delta.
  \end{displaymath}
  This implies the H\"older continuity of the map~\eqref{eq:26} taking
  values in $\mathbf{L}_{\infty}$ on $(0,1)$. As, for every $x\in
  \mathbb{R}^d$, the function $u_c(\cdot,e^{h(\cdot,x)},x)$ on $[0,1]$
  is continuous, we also obtain the H\"older continuity of this map on
  $[0,1]$.
\end{proof}

The following lemma accomplishes the first step. For $w\in \Sigma^M$
denote
\begin{displaymath}
  \eta(w) \set e^{-\int_0^1
    r(t,X_t)dt} U_c(\Lambda,X_1;w)\Lambda
  + \int_0^1 e^{-\int_0^t r(s,X_s)ds}
  u_c(t,\lambda_t,X_t;w)\lambda_t dt. 
\end{displaymath}
Recall the functions $(\Phi^m)$ and the set $\mathcal{W}$ defined in
Section~\ref{sec:main-results}.

\begin{Lemma}
  \label{lem:11}
  An Arrow-Debreu equilibrium exists.  We have
  \begin{align*}
    \mathbb{E}\bigl[\sup_{w\in \Sigma^M} \eta(w)\bigr]<\infty
  \end{align*}
  and, hence, the functions $(\Phi^m)$ are well-defined.  The set
  $\mathcal{W}$ is not empty and belongs to the interior of
  $\Sigma^M$. The set of Arrow-Debreu equilibria is given by $(yP(w),
  (C^m(w))_{m=1,\dots,M})$, $y>0$, $w\in \mathcal{W}$, where the
  consumptions $C^m(w)$ are defined in~\eqref{eq:12} and the
  consumption price process $P(w)$ is given by
  \begin{equation}
    \label{eq:29}
    \begin{split}
      P_t(w) &\set e^{-\int_0^t r(s,X_s)ds}
      u_c(t,\lambda_t,X_t;w) \ind{t<1} \\
      &\quad + e^{-\int_0^1 r(s,X_s)ds} U_c(\Lambda,X_1;w) \ind{t=1},
      \; t\in [0,1].
    \end{split}
  \end{equation}
\end{Lemma}
\begin{proof}
  We use the criteria for the existence of Arrow-Debreu equilibria
  from \cite{Kram:13}. Theorem~1.4, Corollary~1.5, Theorem~2.4, and
  Lemma~3.1 in this paper imply the assertions of the lemma if, for
  $z\in (0,1]$ and $m=1,\dots,M$,
  \begin{displaymath}
    \mathbb{E}\Bigl[\int_0^1 e^{-\int_0^t
      r(s,X_s)ds} \abs{u^m(t,z\lambda_t,X_t)} dt + e^{-\int_0^1 r(s,X_s)ds}
    \abs{U^m(z\Lambda,X_1)} \Bigr] < \infty.
  \end{displaymath}
 
  From the estimates~\eqref{eq:9} and~\eqref{eq:10} for $u^m$
  and~\eqref{eq:25} for $U^m$ and the linear growth conditions for $H$
  and $h$ we deduce the existence of a constant $N>0$ such that for
  every $0<z\leq 1$
  \begin{align*}
    \abs{u^m(t,z\lambda_t,X_t)} + \abs{U^m(z\Lambda,X_1)} &=
    \abs{u^m(t,ze^{h(t,X_t)},X_t)}+
    \abs{U^m(ze^{H(X_1)},X_1)}\\
    &\leq \frac1{z^N}e^{N(1+ \sup_{t\in [0,1]}\abs{X_t})}.
  \end{align*}
  The result now follows from the boundedness of the impatience rate
  function $r=r(t,x)$ and from the existence of all exponential
  moments for $\sup_{t\in [0,1]}\abs{X_t}$.
\end{proof}

Hereafter, for $w\in \Sigma^M$, we denote by $P(w)$ the consumption
price process defined in~\eqref{eq:29}.

\begin{Lemma}
  \label{lem:12}
  For every $w\in \Sigma^M$ we have
  \begin{displaymath}
    \mathbb{E}[P_1(w) \Psi] = \mathbb{E}[e^{-\int_0^1 r(t,X_t)dt}
    U_c(\Lambda,X_1;w) \Psi] < \infty;  
  \end{displaymath}
  in particular, the martingale $Y(w)$ is well-defined.
\end{Lemma}
\begin{proof}
  From Lemma~\ref{lem:8} we deduce the existence of $N>0$ such that
  \begin{displaymath}
    \abs{U^m(e^y,x)} \leq e^{N(1+\abs{x}+\abs{y})}, \; (x,y)\in
    \mathbb{R}^d\times \mathbb{R}, 
  \end{displaymath}
  and, then, from Theorem~\ref{th:2} that $U(w)=U(c,x;w)$
  satisfies~\ref{item:4}.  Lemma~\ref{lem:8}, the linear growth of
  $H=H(x)$, and the boundedness of $r=r(t,x)$ imply the existence of
  $N>0$ such that
  \begin{displaymath}
    \abs{P_1(w)} = U_c(e^{H(X_1)},X_1) e^{-\int_0^1 r(t,X_t)dt} \leq e^{N(1+
      \abs{X_1})}. 
  \end{displaymath}
  From~\ref{item:2} we deduce a similar estimate for $\Psi$:
  \begin{displaymath}
    \Psi = G(X_1) e^{\int_0^1 q(s,X_s)ds} \leq e^{N(1+
      \abs{X_1})}.
  \end{displaymath}
  The integrability of $P_1(w)\Psi$ follows now from the existence of
  all exponential moments for $\abs{X_1}$.
\end{proof}

Lemmas~\ref{lem:1}, \ref{lem:11}, and~\ref{lem:12} imply that the set
of complete Radner equilibria is a subset of $((B(w),S(w)),C(w))$,
$w\in \mathcal{W}$. In view of Lemma~\ref{lem:2}, it only remains to
be shown that for every $w\in \mathcal{W}$ the $(B(w),S(w))$-market is
complete.  This is accomplished in
\begin{Lemma}
  \label{lem:13}
  For every $w\in \Sigma^M$ the $\mathbb{Q}(w)$-martingale $S(w) =
  (S^j_t(w))$ is well-defined and $\mathbb{Q}(w)$ is its only
  equivalent martingale measure.
\end{Lemma}
\begin{proof}
  We rely on the martingale representation result from
  \cite{KramPred:14}. Recall that
  \begin{align*}
    \frac{d\mathbb{Q}(w)}{d\mathbb{P}} &= \frac{Y_1(w)}{Y_0(w)}, \\
    Y_1(w) &= K(X_1;w) e^{\int_0^1 \beta(t,X_t) dt}, \\
    S_1^j(w) &= F^j(X_1)e^{\int_0^1 \alpha^j(t,X_t) dt} + \int_0^1
    \frac{g^j(t,X_t;w)}{Y_t(w)} e^{\int_0^t (\alpha^j+\beta)(s,X_s)ds}
    dt,
  \end{align*}
  where
  \begin{align*}
    \beta(t,x) &= q(t,x) - r(t,x),\\
    \alpha^j(t,x) &= p^j(t,x) - q(t,x), \\
    K(x;w) &= G(x) U_c(e^{H(x)},x;w), \\
    g^j(t,x;w) &= f^j(t,x) u_c(t,e^{h(t,x)},x;w).
  \end{align*}
  According to Theorem~2.3 in \cite{KramPred:14} the result follows if
  the diffusion coefficients $b$ and $\sigma$ satisfy~\ref{item:1} and
  there is a constant $N=N(w)>0$ such that
  \begin{enumerate}[label=(\roman{*}), ref=(\roman{*})]
  \item \label{item:10} The functions $F^j$ and $K$ are weakly
    differentiable, $K$ is strictly positive, the Jacobian matrix
    $\left(F^j_{x^i}\right)_{i=1,\dots,d, \; j=1,\dots,J}$ has rank
    $d$ almost surely under the Lebesgue measure on $\mathbb{R}^d$,
    and
    \begin{equation}
      \label{eq:30}
      \abs{F^j_{x^i}(x)} + \abs{K_{x^i}(x;w)} \leq e^{N(1 + \abs{x})}, \;
      x\in \mathbb{R}^d.
    \end{equation}
  \item \label{item:11} The maps $t\mapsto e^{-N \abs{\cdot}}
    g^j(t,\cdot;w) \set \bigl(e^{-N \abs{x}} g^j(t,x;w)\bigr)_{x\in
      \mathbf{R}^d}$, $t\mapsto \alpha^j(t,\cdot)$, and $t\mapsto
    \beta(t,\cdot)$ of $[0,1]$ to $\mathbf{L}_{\infty}$ are analytic
    on $(0,1)$ and H\"older continuous on $[0,1]$.
  \end{enumerate}

  The required properties of $F^j$, $\alpha^j$, and $\beta$ follow
  immediately from the conditions of the theorem.

  Clearly, $K>0$. By chain rule,
  \begin{align*}
    K_{x^i}(x;w) &= G_{x^i}(x) U_{c}(e^{H(x)},x;w) \\& + G(x)\left(
      U_{cc}(e^{H(x)},x;w) e^{H(x)} H_{x^i}(x) +
      U_{cx^i}(e^{H(x)},x;w)\right).
  \end{align*}
  From~\ref{item:2} we deduce the existence of a constant $N>0$ such
  that
  \begin{displaymath}
    \abs{G_{x^i}(x)} + G(x) \leq e^{N(1+\abs{x})}, \; x\in
    \mathbb{R}^d. 
  \end{displaymath}
  Theorem~\ref{th:2} implies that $U(w)$ satisfies~\ref{item:4}. Then,
  by Lemma~\ref{lem:8}, there is $N>0$ such that
  \begin{displaymath}
    (\abs{U_c} + \abs{U_{cc}} + \abs{U_{cx^i}})(e^y,x;w) \leq
    e^{N(1+\abs{x}+\abs{y})}, \; (x,y)\in \mathbb{R}^d\times \mathbb{R}.  
  \end{displaymath}
  These inequalities and the condition~\ref{item:5} for $H$ imply the
  exponential estimate~\eqref{eq:30} for $K_{x^i}$.

  Theorem~\ref{th:2} implies that $u(w) = u(t,x,;w)$ satisfies same
  conditions~\ref{item:6} or~\ref{item:8} as $u^m$. The
  assertion~\ref{item:11} for $g^j$ follows now from
  Lemmas~\ref{lem:9} and~\ref{lem:10} and the properties of $f^j$
  in~\ref{item:2}.

  A careful reader may notice that in Lemma~\ref{lem:10}, in addition
  to~\ref{item:8}, we assumed the continuity of $u_c(\cdot,c,x)$ on
  $[0,1]$. This does not restrict any generality. Indeed, denoting
  $q\set u_{ct}/u_c$ we obtain that
  \begin{displaymath}
    u_c(t,c,x) = u_c(1/2,c,x)e^{\int_{1/2}^t q(r,c,x)dr}
  \end{displaymath}
  and, as $q$ is bounded, $u_c(\cdot,c,x)$ can be continuously
  extended on $[0,1]$.
\end{proof}

\appendix\normalsize

\section{On analytic functions with values in $\mathbf{L}_\infty$}
\label{sec:analytic-maps}

In this appendix we state versions of composition and implicit
function theorems for analytic functions with values in
$\mathbf{L}_\infty$ used in the proofs of Theorems~\ref{th:1}
and~\ref{th:2}. Hereafter, $\mathbf{L}_\infty\set\mathbf{L}_\infty(E)$
for some ${F}_\sigma$-set $E\subset \mathbb{R}^d$.

\begin{Theorem}
  \label{th:3}
  Let $f = f(x,y)$ and $g_i = g_i(x,z)$, $i=1,\dots,m$, where $x\in
  \mathbb{R}^n$, $y\in \mathbb{R}^m$, and $z\in \mathbb{R}^l$, be
  analytic maps of neighborhoods of $(0,0)$ to $\mathbf{L}_\infty$
  such that $g_i(0,0)=0$.  Then
  \begin{displaymath}
    h(x,z) \set f(x,g_1(x,z),\dots,g_m(x,z)), \; (x,z)\in
    \mathbb{R}^n\times \mathbb{R}^l,
  \end{displaymath}
  is an analytic map of a neighborhood of $(0,0)$ to
  $\mathbf{L}_\infty$.
\end{Theorem}

\begin{Theorem}
  \label{th:4}
  Let $f = f(x,y)$, where $x\in \mathbb{R}^n$ and $y\in \mathbb{R}$,
  be an analytic map of a neighborhood of $(0,0)$ to
  $\mathbf{L}_\infty$ such that $f(0,0)=0$ and $f_y(0,0) \geq
  \epsilon$ for some constant $\epsilon>0$. Then there is a
  neighborhood $V\subset \mathbb{R}^n$ of $0$ and an analytic map
  $g=g(x)$ of $V$ to $\mathbf{L}_\infty$ such that
  \begin{displaymath}
    f(x,g(x))=0, \; x\in V. 
  \end{displaymath}
  
\end{Theorem}

The proofs of both theorems are essentially identical to the proofs of
the corresponding results for real-valued analytic functions, see,
e.g., Propositions~IV.5.5.1 and~IV.5.6.1 in \cite{Cart:95}.

\bibliographystyle{plainnat}

\bibliography{../bib/finance}

\end{document}